\documentclass[twocolumn]{autart}
\makeatletter
\def\ps@copyright{%
  \let\@mkboth\@gobbletwo
  \def\@oddhead{}%
  \let\@evenhead\@oddhead
  \def\@oddfoot{\hfil{\rmfamily\thepage}\hfil}%
  \let\@evenfoot\@oddfoot
}
\makeatother

\usepackage{amsmath}
\usepackage{amssymb}
\usepackage{enumitem}
\usepackage{array}
\usepackage[colorlinks = true, linkcolor = blue, citecolor = blue, urlcolor = blue]{hyperref}
\usepackage[nameinlink]{cleveref}
\usepackage{xcolor}
\usepackage{graphicx}

\newtheorem{algorithmo}{Algorithm}
\newtheorem{condition}{Condition}
\newtheorem{corollary}{Corollary}
\newtheorem{definition}{Definition}
\newtheorem{problem}{Problem}
\newtheorem{proposition}{Proposition}
\newtheorem{remark}{Remark}
\newtheorem{theorem}{Theorem}

\crefname{equation}{}{}
\crefname{condition}{Condition}{Conditions}

\begin{document}

\begin{frontmatter}

\title{Global exponential stabilization of a force- and \\ torque-actuated unicycle by flexible-step MPC\thanksref{footnoteinfo}}
\thanks[footnoteinfo]{This paper was not presented at any IFAC meeting. Corresponding author: Ala Kolsi}
\author[CE]{Ala Kolsi}\ead{ala.kolsi@ic.rwth-aachen.de},
\author[CE]{Christian Ebenbauer}\ead{christian.ebenbauer@ic.rwth-aachen.de},
\author[BG]{Bahman Gharesifard}\ead{bahman.gharesifard@queensu.ca},
\author[CE]{Raik Suttner}\ead{raik.suttner@ic.rwth-aachen.de}

\address[CE]{Chair of Intelligent Control Systems, RWTH Aachen University, Aachen, Germany}
\address[BG]{Department of Mathematics and Statistics at Queen's University, Kingston, Canada}

\begin{abstract}
We study the problem of global exponential stabilization of a force- and torque-controlled unicycle model in discrete time. To this end, we extend a recently introduced approach to model predictive control (MPC) in which a flexible number of inputs is implemented in every iteration. We present the first flexible-step MPC protocol with state-dependent weights for average descent. Notably, the proposed method relies neither on a suitable design of running or terminal cost functions nor on a suitable choice of terminal constraints. Instead, stability is guaranteed through a generalized discrete-time control Lyapunov function. We establish a new theoretical framework for global exponential stabilization of general nonlinear discrete-time control systems by flexible-step MPC. The obtained results go beyond the unicycle example. However, given the importance of the unicycle dynamics, we make that a focal point of our work. For the particular case of the dynamic (second-order) unicycle model, we show that global exponential stability cannot be attained in the classical sense, but in a slightly weaker sense. The proposed flexible-step MPC method is shown to induce the best possible notion of global exponential stability for this model. We provide explicit rules for the choice of parameters, which guarantee feasibility and global exponential stability. Our numerical simulations show that the discrete MPC method also works very well in applications to a continuous-time torque-actuated unicycle.
\end{abstract}

\begin{keyword}
Model predictive control, Lyapunov methods, Exponential stabilization, Nonholonomic systems.
\end{keyword}

\end{frontmatter}

%----------------------------------------------------------------------
%
%                               Section 1
%
%----------------------------------------------------------------------

\section{Introduction}\label{sec:1}
Wheeled mobile robots are widely used for industrial and service applications due to their mechanical simplicity and versatility. The unicycle model is a standard benchmark example, which serves as a simplified model for the motion of a mobile robot. Many studies assume that the longitudinal and rotational velocity can be set instantaneously through the inputs. However, when inertial effects cannot be neglected, a force- and torque-actuated model may provide a more realistic description. In the present paper, we consider a discrete-time dynamic unicycle model in which the velocities do not serve as inputs but are part of the system state.

A frequently studied control objective in this context is asymptotic stabilization about a prescribed position and orientation. In case of nonholonomic velocity constraints, this task can typically not be accomplished by continuous static state feedback. To overcome this obstruction, many different methods and ideas have been proposed, including discontinuous static state feedback \cite{Astolfi1996}, 
\cite{Bloch1996}, \cite{Mcloskey1997}, and time-varying continuous state feedback \cite{Samson1993}, \cite{Pomet1992}, 
\cite{Sordalen1995} for the unicycle and related nonholonomic systems.

In many applications, the control objective is not limited to mere asymptotic stabilization, but it also involves satisfaction of constraints and optimization of performance. To meet all of these requirements, model predictive control (MPC) provides a general and well-established approach \cite{GrueneBook}, \cite{Schwenzer2021}. Several MPC-based results on set-point stabilization of unicycles have been reported so far, e.g.~in \cite{Xie2008}, \cite{Zhang2021}, \cite{VanEssen2001}. Notably, the MPC method in \cite{Worthmann2016} provides local asymptotic stability for a kinematic unicycle through a suitable choice of the running cost and the prediction horizon, but without the need for stabilizing terminal costs or constraints. A more general study on MPC without stabilizing terminal costs or constraints, which is based on homogeneous approximations, can be found in \cite{Coron2020}. For the unicycle model, which is non-homogeneous, such a homogeneous approximation can typically only lead to local or semi-global stabilization. Building on the results in \cite{Worthmann2016} and \cite{Coron2020}, semi-global asymptotic stability of a partially dynamic unicycle model has been shown in \cite{Russwurm2021}, where the domain of attraction grows with increasing prediction horizon. In the present paper, we aim for global exponential stability.

In most of the known discrete-time MPC schemes, only the first element of a predicted optimal input sequence is applied to a control system. This procedure typically yields \emph{static state feedback} in the sense that the next applied input depends solely on the current system state. In the present paper, we use an MPC scheme with a temporal component, which generates \emph{time-varying state feedback}. Such an approach is well-suited for applications to control systems that cannot be asymptotically stabilized by smooth static state feedback, such as nonholonomic systems. The temporal component of the employed MPC schemes originates from a flexible number of implemented steps for each predicted optimal input sequence. Here, we build on the recently introduced flexible-step MPC approach from \cite{Furnsinn20251}. The distinct feature of this approach relative to other MPC methods with a flexible number of implemented inputs, see \cite{Yang1993} \cite{Alamir2017} \cite{Worthmann2017}, is that the tasks of optimization and stabilization are separated. This means that a suitable design of cost functions or terminal constraints for stabilization, which is often difficult to obtain, is not needed. Instead, stability of a system is guaranteed through a so-called \emph{generalized control Lyapunov function} in the underlying optimal control problem. Relaxed feasibility and stability criteria for flexible-step MPC are presented in \cite{Furnsinn2023}. Also applications to known switched linear systems \cite{Furnsinn20252} and unknown linear systems \cite{Pietschner2026} have been investigated so far. In the present paper, we investigate exponential stabilization of nonlinear systems by flexible-step MPC with a focus on the dynamic (second-order) unicycle model.

\emph{Main contributions}. We extend the theory of flexible-step MPC from \cite{Furnsinn20251} to a general framework for global exponential stabilization of nonlinear discrete-time systems. We relax several components of flexible-step MPC so that very simple generalized control Lyapunov functions can be used for stabilization. For the example of a dynamic unicycle, we show that choosing the state norm as this function is sufficient for global exponential stability. In this context, exponential stability in the classical sense will be shown to be not attainable for the dynamic unicycle model. However, one can achieve this by relaxing the exponent on the norm of the initial state. Remarkably, it turns out that the proposed flexible-step MPC protocol achieves the best possible exponent. To the best of our knowledge, this is also the first result on global exponential stabilization of a force- and torque-actuated unicycle by model predictive control.

The paper is organized as follows. Basic notation and definitions are summarized in \Cref{sec:2}. The problem of exponential stabilization of a dynamic unicycle model and its challenges are described in \Cref{sec:3}. An approach to global exponential stabilization of general nonlinear discrete-time systems by flexible-step MPC and a general stability result are presented in \Cref{sec:4}. The general theory is then applied to the unicycle model in \Cref{sec:5}. Finally, numerical tests can be found in \Cref{sec:6}, where the proposed discrete-time MPC scheme is applied to a continuous-time force- and torque-actuated unicycle.

%----------------------------------------------------------------------
%
%                               Section 2
%
%----------------------------------------------------------------------

\section{Notation}\label{sec:2}
Let $\mathbb{N}$ denote the set of nonnegative integers. For $s,t\in\mathbb{N}$ with $s\leq{t}$, the integer intervals $[s\!:\!t]$, $[s\!:\!t)$, $(s\!:\!t]$, and $(s\!:\!t)$ are given by the intersection of $\mathbb{N}$ with the usual real number intervals $[s,t]$, $[s,t)$, $(s,t]$, and $(s,t)$, respectively. Let $\mathcal{K}$ denote the set of continuous and strictly increasing functions $\gamma\colon[0,\infty)\to[0,\infty)$ with $\gamma(0)=0$. Let $\mathcal{K}_\infty$ be the set of all elements $\gamma$ of $\mathcal{K}$ with $\gamma(r)\to\infty$ as $r\to\infty$. Let $\mathcal{KL}$ be the set of continuous functions $\beta\colon[0,\infty)^2\to[0,\infty)$ such that $\beta(\cdot,t)$ is an element of $\mathcal{K}$ for each $t\geq0$ and such that $\beta(r,\cdot)$ is non-increasing with $\beta(r,t)\to0$ as $t\to\infty$ for each $r\geq0$. For implemented states and inputs of a control system at time $t\in\mathbb{N}$, we use the bracket notation $x(t)$ and $u(t)$. For predicted states and inputs of an optimal control problem at time $i\in\mathbb{N}$ within the prediction horizon, we use the subscript notation $x_i$ and $u_i$.

%----------------------------------------------------------------------
%
%                               Section 3
%
%----------------------------------------------------------------------

\section{Problem statement and challenges}\label{sec:3}
We consider a dynamic (force- and torque-controlled) unicycle model with mass $m>0$, inertia $J>0$, and linear damping constants $k,\kappa>0$. In continuous time, the second-order dynamic unicycle model reads%
\begin{align}\label{eq:unicycleModelContinuousTime}%
& \dot{\mathrm{x}} \ = \ v\,\cos(\theta), \qquad \dot{\mathrm{y}} \ = \ v\,\sin(\theta), \qquad \dot{\theta} \ = \ \omega, \nonumber \\
& m\,\dot{v} \ = \ F-k\,v, \qquad  J\,\dot{\omega} \ = \ \mathcal{T}-\kappa\,\omega,
\end{align}%
see~e.g.~\cite{Gazi2007}. The Euler discretization of the continuous-time system \cref{eq:unicycleModelContinuousTime} with step size $h>0$ is given by%
\begin{subequations}\label{eq:unicycleModel}%
\begin{align}
\mathrm{x}(t+1) & \ = \ \mathrm{x}(t) + h\,v(t)\,\cos(\theta(t)), \label{eq:unicycleModel:a} \allowdisplaybreaks \\
\mathrm{y}(t+1) & \ = \ \mathrm{y}(t) + h\,v(t)\,\sin(\theta(t)), \label{eq:unicycleModel:b} \allowdisplaybreaks \\
\theta(t+1) & \ = \ \theta(t) + h\,\omega(t), \label{eq:unicycleModel:c} \allowdisplaybreaks \\
v(t+1) & \ = \ v(t) + \tfrac{h}{m}\,(F(t)-k\,v(t)), \label{eq:unicycleModel:d} \allowdisplaybreaks \\
\omega(t+1) & \ = \ \omega(t) + \tfrac{h}{J}\,(\mathcal{T}(t)-\kappa\,\omega(t)). \label{eq:unicycleModel:e}
\end{align}%
\end{subequations}%
The five-dimensional system state of \cref{eq:unicycleModel} consists of the position components $\mathrm{x}(t),\mathrm{y}(t)$, an angle $\theta(t)$ for the alignment, the linear velocity $v(t)$, and the angular velocity $\omega(t)$. The two input components of \cref{eq:unicycleModel} are a force $F(t)$ and a torque $\mathcal{T}(t)$. We naturally use the notation%
\begin{subequations}\label{eq:unicycleNotation}%
\begin{align}
x(t) & \ = \ \big[ \ \mathrm{x}(t),\ \mathrm{y}(t),\ \theta(t),\ v(t),\ \omega(t) \ \big]^\top \in \mathbb{R}^5, \\
u(t) & \ = \ \big[ \ F(t),\ \mathcal{T}(t) \ \big]^\top \in \mathbb{R}^2
\end{align}%
\end{subequations}%
for the state and the input vector of \cref{eq:unicycleModel}, respectively. The control objective is to globally exponentially stabilize the unicycle model \cref{eq:unicycleModel} with respect to the origin.

The nonlinear discrete-time control system \cref{eq:unicycleModel} does not satisfy Brockett's necessary condition for local asymptotic stabilization from \cite{Sundarapandian2002}. Hence, there exists no continuously differentiable static state feedback that locally asymptotically stabilizes \cref{eq:unicycleModel} around the origin. However, system \cref{eq:unicycleModel} can be asymptotically stabilized by time-varying state feedback. The proposed flexible-step MPC approach in this paper may be seen as such a time-varying method. It differs from the standard one-step implementation of MPC, which typically generates static state feedback. A flexible number of implemented steps for each predicted optimal input sequence represents an additional temporal component, which is beneficial for asymptotic stabilization. 
The proposed general flexible-step MPC approach in the subsequent \Cref{sec:4} is not limited to the unicycle model \cref{eq:unicycleModel} but also applicable to discretizations of other nonholonomic systems.

In the present paper, we do not limit our objectives to local asymptotic stabilization of \cref{eq:unicycleModel}. Instead, we aim to achieve global exponential stabilization. However, this goal comes with an additional challenge. Namely, the question whether it is possible to globally exponentially stabilize \cref{eq:unicycleModel} by any kind of control strategy. As we will see below, the answer is negative if we want to achieve the classical notion of exponential stability. The subsequent ``negative result,'' \Cref{prop:noGES}, makes a stronger statement: Even a weaker stability property than local exponential stability cannot be achieved for \cref{eq:unicycleModel}. There is no control strategy for the unicycle model \cref{eq:unicycleModel} such that, locally around the origin, every solution $x\colon\mathbb{N}\to\mathbb{R}^5$ of the closed-loop system satisfies an exponential stability condition of the form%
\begin{equation}\label{eq:MichelBook3}
\|x(t)\| \ \leq \ \|x(0)\|^q\,\lambda\,\mathrm{e}^{-\mu\,t},
\end{equation}%
where $\lambda,\mu,q$ are positive real numbers with $1/2<q\leq1$ and $\|\cdot\|$ denotes the Euclidean norm on $\mathbb{R}^5$, we state this formally next.%
\begin{proposition}\label{prop:noGES}
Let $r$, $q$, $\lambda$, and $\mu$ be positive real numbers with $1/2<q\leq1$. Then there exists $\varepsilon\in(0,r)$ such that, for every $u\colon\mathbb{N}\to\mathbb{R}^2$, there exists $\tau\in\mathbb{N}$ such that the solution $x\colon\mathbb{N}\to\mathbb{R}^5$ of \cref{eq:unicycleModel} with input $u$ and initial condition%
\begin{equation}\label{eq:noGES}
x(0) \ = \ \big[ \ 0, \ \varepsilon, \ 0, \ 0, \ 0 \ \big]^\top
\end{equation}%
violates inequality \cref{eq:MichelBook3} at time $t=\tau$.
\end{proposition}%
A proof of \Cref{prop:noGES} can be found in \hyperlink{sec:A}{Appendix~A}.

Because of \Cref{prop:noGES}, the ``best'' exponent $q$ that one can expect for \cref{eq:unicycleModel} in a local exponential stability condition of the form \cref{eq:MichelBook3} is $q=1/2$. The exponent $q=1/2$ is directly linked to a local homogeneity property of the dynamic unicycle model, which originates from the product $v(t)\,\sin(\theta(t))$ in \cref{eq:unicycleModel} when the angle $\theta(t)$ and the velocity $v(t)$ are close to zero. A local exponential stability property of the form \cref{eq:MichelBook3} with $q=1/2$ is weaker than the classical notion with $q=1$, since it allows more overshoot for solutions that start close to the origin. This issue is also reflected by the initial condition \cref{eq:noGES} in \Cref{prop:noGES}. For small $\varepsilon>0$, initial condition \cref{eq:noGES} prohibits a fast ``zig-zag motion'' of the unicycle towards the origin with large velocities. This problem does not occur for the more frequently studied kinematic unicycle model \cref{eq:unicycleModel:a}-\cref{eq:unicycleModel:c} in which the velocities $v(t)$ and $\omega(t)$ are not components of the system state but inputs. Since \Cref{prop:noGES} rules out classical exponential stability for \cref{eq:unicycleModel}, we introduce the following potentially weaker notion of exponential stability.%
\begin{definition}\label{def:GES}
Let $\mathbb{X}$ be a finite-dimensional vector space. Let $\|\cdot\|$ be a norm on $\mathbb{X}$. Let $g\colon\mathbb{N}\times\mathbb{X}\to\mathbb{X}$ and $\gamma\in\mathcal{K}_\infty$. We say that the time-varying discrete-time system%
\begin{equation}\label{eq:MichelBook1}
x(t+1) \ = \ g(t,x(t))
\end{equation}%
is \emph{globally $\gamma$-exponentially stable} if there exist positive real numbers $\lambda$ and $\mu$ such that%
\begin{equation}\label{eq:MichelBook2}
\|x(t)\| \ \leq \ \gamma(\|x(0)\|)\,\lambda\,\mathrm{e}^{-\mu\,t}
\end{equation}%
for every solution $x$ of \cref{eq:MichelBook1} and every time $t\in\mathbb{N}$.
\end{definition}%
\Cref{def:GES} is inspired by the notion of \emph{$\mathcal{K}$-exponential stability} from \cite{Sordalen1995}. For $\gamma=\mathrm{id}$, inequality~\cref{eq:MichelBook2} reduces to the usual condition for global exponential stability. Because of \Cref{prop:noGES}, $\gamma(r)=\mathcal{O}(r^{1/2})$ for $r\to0$ is the ``best'' local asymptotic behavior of $\gamma\in\mathcal{K}_\infty$ that we can expect for the unicycle model \cref{eq:unicycleModel}. We will show in \Cref{sec:5} that this best-case scenario is actually established by flexible-step MPC. More precisely, we will prove global $\gamma$-exponential stability for the closed-loop system, where $\gamma\in\mathcal{K}_\infty$ is defined by%
\begin{equation}\label{eq:max}
\gamma(r) \ := \ \max\{r^{1/2},r\}.
\end{equation}%
That is, flexible-step MPC achieves the best possible local exponential stability property and also classical exponential stability for all solutions with initial states outside an arbitrary small neighborhood of the origin.

In the subsequent \Cref{sec:4}, a new approach to global exponential stabilization by flexible-step MPC for general nonlinear discrete-time control systems is presented. Later, in \Cref{sec:5}, the general approach is applied to the unicycle model \cref{eq:unicycleModel} and global $\gamma$-exponential stability (as in \Cref{def:GES}) is proved for $\gamma\in\mathcal{K}_\infty$ defined by \cref{eq:max}.

%----------------------------------------------------------------------
%
%                               Section 4
%
%----------------------------------------------------------------------

\section{Global exponential stabilization of nonlinear systems by flexible-step MPC}\label{sec:4}
In this section, we present a flexible-step MPC method for global exponential stabilization, which is not only applicable to the unicycle model \cref{eq:unicycleModel}, but to a more general class of nonlinear discrete-time control systems. The general approach in this section is applied to the unicycle model \cref{eq:unicycleModel} later in \Cref{sec:5}.

Throughout this section, let $\mathbb{X}$ and $\mathbb{U}$ be finite-dimen\-sional vector spaces and let $\|\cdot\|$ be a norm on $\mathbb{X}$. Let $f$ be a map from $\mathbb{X}\times\mathbb{U}$ to $\mathbb{X}$. We consider the nonlinear discrete-time control system%
\begin{equation}\label{eq:generalControlSystem}
x(t+1) \ = \ f(x(t),u(t))
\end{equation}%
with state $x(t)\in\mathbb{X}$ and input $u(t)\in\mathbb{U}$ at time $t\in\mathbb{N}$. Clearly, the unicycle model \cref{eq:unicycleModel} in the previous section is a special case of \cref{eq:generalControlSystem}. In the next paragraph, we introduce the key components of flexible-step MPC for global $\gamma$-exponential stability in the sense of \Cref{def:GES}. This involves a so-called generalized control Lyapunov function for \cref{eq:generalControlSystem} with suitable properties.
The subsequent \Cref{def:EGDCLF} provides a novel class of these control Lyapunov functions, which is tailored for the goal of establishing global $\gamma$-exponential stability in the sense of \Cref{def:GES}.%
\begin{definition}\label{def:EGDCLF}
Let $N$ be a positive integer. By an \emph{exponential generalized discrete-time control Lyapunov function (EGDCLF) of order $N$ for \cref{eq:generalControlSystem}} we mean a real-valued function $V$ on $\mathbb{X}$ together with a constant $\alpha\in(0,1)$ and functions $\sigma_1,\ldots,\sigma_N\colon\mathbb{X}\to(0,\infty)$ such that the following four properties \ref{def:EGDCLF:1}--\ref{def:EGDCLF:4} hold.\vspace{-0.3cm}

\begin{enumerate}[label=(P\arabic*),leftmargin=0.8cm]
	\item\label{def:EGDCLF:1} There exist $\chi_1,\chi_2\in\mathcal{K}_\infty$ such that%
	\begin{equation}\label{eq:EGDCLF:1}
	\chi_1(\|x\|) \ \leq \ V(x) \ \leq \ \chi_2(\|x\|)
	\end{equation}%
	for every $x\in\mathbb{X}$.%
	\item\label{def:EGDCLF:2} For every $x_0\in\mathbb{X}$, there exists $u\colon[\vphantom{]}0\!:\!N\vphantom{(})\to\mathbb{U}$ such that the solution $x\colon[0\!:\!N]\to\mathbb{X}$ of \cref{eq:generalControlSystem} with input $u$ and initial condition $x(0)=x_0$ satisfies the \emph{average descent condition}%
	\begin{equation}\label{eq:EGDCLF:2}
	\sum_{t=1}^N\sigma_t(x_0)\,V(x(t)) \ \leq \ (1-\alpha)\,V(x_0),
	\end{equation}%
	where%
	\begin{equation}\label{eq:EGDCLF:3}
	\sigma_1(x_0) + \cdots + \sigma_N(x_0) \ \geq \ 1.
	\end{equation}%
	\item\label{def:EGDCLF:3} There exists $\varphi\in\mathcal{K}_\infty$ such that%
	\begin{equation}\label{eq:EGDCLF:4}
	(1-\alpha)\,V(x) \ \leq \ \varphi(V(x))\cdot\min_{i\in[1:N]}\sigma_i(x)
	\end{equation}%
	for every $x\in\mathbb{X}$.%
	\item\label{def:EGDCLF:4} There exist $\tilde{\lambda},\tilde{\mu}>0$ such that the $\mathcal{K}_\infty$-functions $\chi_1$ from \ref{def:EGDCLF:1} and $\varphi$ from \ref{def:EGDCLF:3} satisfy%
	\begin{equation}\label{eq:EGDCLF:5}
	(\chi_1^{-1}\circ\varphi)(r\,\mathrm{e}^{-\tau}) \ \leq \ (\chi_1^{-1}\circ\varphi)(r)\,\tilde{\lambda}\,\mathrm{e}^{-\tilde{\mu}\,\tau}
	\end{equation}%
	for every $r\geq0$ and every $\tau\geq-\alpha$.
\end{enumerate}%
\end{definition}%
Properties \ref{def:EGDCLF:1}--\ref{def:EGDCLF:4} in \Cref{def:EGDCLF} serve the following purposes. Property \ref{def:EGDCLF:1} is just the standard requirement of a global control Lyapunov function to be positive definite and radially unbounded. Property \ref{def:EGDCLF:2} guarantees the existence of inputs over the \emph{prediction horizon} $N$ for which $V$ decays on average along solutions of \cref{eq:generalControlSystem}. Note that the average descent condition \cref{eq:EGDCLF:2} does not demand monotonic descent of $V$ in every step. It only requires that a weighted average of $V$ is reduced relative to the initial value. In the stability analysis of the closed-loop system (see proof of \Cref{thm:expStab} in \hyperlink{sec:B}{Appendix~B}), the EGDCLF $V$ will play the role of a \emph{non-monotonic Lyapunov function} (see e.g.~\cite{MichelBook}). The functions $\sigma_1,\ldots,\sigma_N$ in \cref{eq:EGDCLF:2} act as state-dependent \emph{weights} of the values of $V$ and the \emph{decay constant} $\alpha$ determines the magnitude of average descent. The state dependence of the weights is one of the main differences relative to the class of generalized control Lyapunov functions in \cite{Furnsinn20251}, where the weights are assumed to be constant. This novel feature will allow us to use very simple and basic EGDCLFs for flexible-step MPC, such as the state norm $V=\|\cdot\|$ for the unicycle model in \Cref{sec:5}. Property \ref{def:EGDCLF:3} is also a novel feature compared to \cite{Furnsinn20251}. It replaces the \emph{small-control property} in \cite{Furnsinn20251}, which can be difficult to check in practice because it requires a careful analysis of the input--state trajectories of \cref{eq:generalControlSystem}. By contrast, property \ref{def:EGDCLF:3} does not involve input--state trajectories of \cref{eq:generalControlSystem}. It provides a suitable upper bound for the values of $V$ in the average descent condition \cref{eq:EGDCLF:2}; see equation \cref{eq:proofExpStab:4} in \hyperlink{sec:B}{Appendix~B}. Finally, property \ref{def:EGDCLF:4} is a technical growth assumption on the comparison functions $\chi_1$ and $\varphi$ in order to get not just asymptotic stability but in fact exponential stability. It guarantees that the composition $\chi_1^{-1}\circ\varphi$ decays exponentially fast whenever its argument decays exponentially fast; see equation \cref{eq:proofExpStab:6} in \hyperlink{sec:B}{Appendix~B}.%
\begin{remark}
The defining properties of an EGDCLF in \Cref{def:EGDCLF} are tailored for the purpose of global exponential stabilization. Since exponential stability is a stronger property than asymptotic stability, some of the assumptions on an EGDCLF (like \ref{def:EGDCLF:2} and \ref{def:EGDCLF:4} in \Cref{def:EGDCLF}) are stronger than the assumptions on a \emph{generalized discrete-time control Lyapunov functions} (GDCLF) in \cite{Furnsinn20251}, where the focus is on asymptotic stabilization. Moreover, the definition of a GDCLF in \cite{Furnsinn20251} is more general than \Cref{def:EGDCLF} in the sense that a GDCLF may not only depend on the state but also on a finite sequence of previously predicted inputs. In the present paper, we only consider purely state-dependent EGDCLFs without an additional dependence on inputs.
\end{remark}%
Now we use the concept of an EGDCLF from \Cref{def:EGDCLF} to guarantee exponential stability for the control system~\cref{eq:generalControlSystem} by means of flexible-step MPC. To this end, we insert an EGDCLF into the constraints of an optimal control problem. We will subsequently consider the following type of optimal control problem.%
\begin{problem}
Let $N$ be a positive integer, $f_0\colon\mathbb{X}\times\mathbb{U}\to\mathbb{R}$, and $F\colon\mathbb{X}\to\mathbb{R}$. Moreover, let $V\colon\mathbb{X}\to\mathbb{R}$, $\alpha\in(0,1)$, and $\sigma_1,\ldots,\sigma_N\colon\mathbb{X}\to(0,\infty)$. For any given initial state $x_0\in\mathbb{X}$, solve the optimal control problem%
\begin{align}
\mathop{\text{minimize}}\limits_{\substack{u_0,\ldots,u_{N-1}\in\mathbb{U}\\x_1,\ldots,x_N\in\mathbb{X}}} \qquad & \sum_{i=0}^{N-1}f_0(x_i,u_i) + F(x_N) \label{eq:generalOCP} \allowdisplaybreaks \\
\text{subject to} \qquad & \text{$\forall{i}\in[0\!:\!N)\colon \ \ x_{i+1}=f(x_i,u_i)$}, \nonumber \\
& \sum_{i=1}^N\sigma_i(x_0)\,V(x_i) \leq (1-\alpha)\,V(x_0), \nonumber
\end{align}%
where it is assumed that if \cref{eq:generalOCP} is \emph{feasible}, i.e, the set over which optimization takes place is not empty, then \cref{eq:generalOCP} has a solution.
\end{problem}%
In this paper, we only consider optimal control problems without hard state and input constraints. Extensions to flexible-step MPC with constraints, as in \cite{Furnsinn20251}, are left to future research. Note, however, that one can use the \emph{running cost} $f_0$ and the \emph{terminal cost} $F$ to enforce soft state and input constraints. In the general stability result (\Cref{thm:expStab}), we do not impose assumptions on $f_0$ and $F$, which means that the user can choose $f_0$ and $F$ freely in any desired manner.

The average descent condition in the last line of \cref{eq:generalOCP} does not guarantee monotonic descent of $V$ in every step, but only for a weighted average of $V$. However, if the weights $\sigma_1,\ldots,\sigma_N$ also satisfy the bound \cref{eq:EGDCLF:3}, then it is possible to choose a ``flexible step'' for which descent of $V$ by a factor $(1-\alpha)$ is guaranteed. For a later reference in the proposed flexible-step MPC algorithm (\Cref{algo:fsMPC}), we make following remark.%
\begin{remark}[choice of a flexible step]\label{rmk:choiceOfFS}
Assume that, for an initial state $x_0\in\mathbb{X}$, the optimal control problem \cref{eq:generalOCP} is feasible and provides optimal inputs $u_0,\ldots,u_{N-1}\in\mathbb{U}$ and optimal states $x_1,\ldots,x_N\in\mathbb{X}$. In addition, assume that inequality \cref{eq:EGDCLF:3} is satisfied. Then, it follows from \cref{eq:EGDCLF:3} and the average descent condition \cref{eq:generalOCP} that there exists a ``flexible step'' $\ell\in[1\!:\!N]$ such that $V(x_\ell)\leq(1-\alpha)\,V(x_0)$. An implementation of the first $\ell$ inputs $u_0,\ldots,u_{\ell-1}$ steers \cref{eq:generalControlSystem} from $x_0$ to $x_\ell$, resulting in a non-monotonic decay of $V$. To induce fast convergence of the system state to the origin, one can select the ``flexible step'' $\ell$ as the first time instant at which the value of $V$ is reduced by the factor $(1-\alpha)$. That is, one chooses the smallest $\ell\in[1\!:\!N]$ for which $V(x_\ell)\leq(1-\alpha)V(x_0)$.
\end{remark}%
Now we can formulate a flexible-step MPC algorithm for the nonlinear discrete-time control system \cref{eq:generalControlSystem}.%
\begin{algorithmo}[flexible-step MPC]\label{algo:fsMPC}$ $\vspace{-0.3cm}

\begin{enumerate}[label=\arabic*:,ref=\arabic*]\setcounter{enumi}{-1}
	\item\label{algo:fsMPC:0} Set $t:=0$.%
	\item\label{algo:fsMPC:1} Measure the current state $x(t)$ of \cref{eq:generalControlSystem}.%
	\item\label{algo:fsMPC:2} Solve the optimal control problem \cref{eq:generalOCP} with initial state $x_0:=x(t)$ and obtain optimal inputs $u_0,\ldots,u_{N-1}\in\mathbb{U}$ and optimal states $x_1,\ldots,x_N\in\mathbb{X}$.%
	\item\label{algo:fsMPC:3} Choose a ``flexible step'' $\ell\in[1\!:\!N]$ such that $V(x_\ell)\leq(1-\alpha)\,V(x_0)$; cf.~\Cref{rmk:choiceOfFS}.%
	\item\label{algo:fsMPC:4} Apply $u(t):=u_k$ to \cref{eq:generalControlSystem} and increment $t\to{t+1}$ for $k=0,1\ldots,\ell-1$. Go to line \ref{algo:fsMPC:1}.
\end{enumerate}%
\end{algorithmo}%
Clearly, for \Cref{algo:fsMPC} to be well-defined, one has to guarantee that the optimal control problem \cref{eq:generalOCP} is always feasible. This is guaranteed by one of the defining properties of an EGDCLF, namely property \ref{def:EGDCLF:2} in \Cref{def:EGDCLF}. The following main result of this section states that if $V$ is an EGDCLF for \cref{eq:generalControlSystem}, then \Cref{algo:fsMPC} is well-defined and the closed-loop system is globally $\gamma$-exponentially stable in the sense of \Cref{def:GES}.%
\begin{theorem}\label{thm:expStab}
Let $N$ be a positive integer, $f_0\colon\mathbb{X}\times\mathbb{U}\to\mathbb{R}$, and $F\colon\mathbb{X}\to\mathbb{R}$. Assume that $V\colon\mathbb{X}\to\mathbb{R}$ is an EGDCLF of order $N$ for \cref{eq:generalControlSystem} with $\alpha\in(0,1)$, $\sigma_1,\ldots,\sigma_N\colon\mathbb{X}\to(0,\infty)$, and $\chi_1,\chi_2,\varphi\in\mathcal{K}_\infty$ as in \Cref{def:EGDCLF}. Then, the optimal control problem \cref{eq:generalOCP} is feasible for every initial state $x_0\in\mathbb{X}$. Moreover, system \cref{eq:generalControlSystem} controlled by \Cref{algo:fsMPC} is globally $\gamma$-exponentially stable, where $\gamma:=\chi_1^{-1}\circ\varphi\circ\chi_2\in\mathcal{K}_\infty$.
\end{theorem}%
A proof of \Cref{thm:expStab} can be found in \hyperlink{sec:B}{Appendix~B}.

To the best of our knowledge, \Cref{thm:expStab} is the first stability result for flexible-step MPC with state-dependent weights $\sigma_1,\ldots,\sigma_N$. As shown in the next section, the state dependence of the weights allows us to consider simple Lyapunov functions, such as the state norm for the unicycle model \cref{eq:unicycleModel}. In contrast to the stability results in \cite{Furnsinn20251} for constant weights, \Cref{thm:expStab} does not require an additional \emph{small control property} of \cref{eq:generalControlSystem} but only the control-independent boundedness property \ref{def:EGDCLF:3} in \Cref{def:EGDCLF}.

%----------------------------------------------------------------------
%
%                               Section 5
%
%----------------------------------------------------------------------

\section{Global exponential stabilization of a unicycle by flexible-step MPC}\label{sec:5}
Now we apply the proposed flexible-step MPC method (\Cref{algo:fsMPC}) to the unicycle model \cref{eq:unicycleModel}. That is, as a special case of the problem in \Cref{sec:4}, we consider system~\cref{eq:generalControlSystem} with state space $\mathbb{X}:=\mathbb{R}^5$, input space $\mathbb{U}:=\mathbb{R}^2$, Euclidean norm $\|\cdot\|$, and, using the notation in \cref{eq:unicycleNotation}, the transition map $f\colon\mathbb{R}^5\times\mathbb{R}^2\to\mathbb{R}^5$ is defined by%
\begin{equation}\label{eq:unicycleMap}
f(x,u) \ := \ \left[\begin{smallmatrix} \mathrm{x} + h\,v\,\cos(\theta) \\ \mathrm{y} + h\,v\,\sin(\theta) \\ \theta + h\,\omega \\ v + \tfrac{h}{m}\,(F - k\,v) \\ \omega + \tfrac{h}{J}\,(\mathcal{T}-\kappa\,\omega) \end{smallmatrix}\right],
\end{equation}%
where $m,J,k,\kappa,h>0$ are arbitrary but fixed constants. As explained in \Cref{sec:3}, the ``best'' notion of global $\gamma$-exponential stability for the unicycle model \cref{eq:unicycleModel} in the sense of \Cref{def:GES} is the one with $\gamma\in\mathcal{K}_\infty$ defined by \cref{eq:max}. For a successful application of \Cref{algo:fsMPC}, we have to choose an EGDCLF $V\colon\mathbb{R}^5\to\mathbb{R}$ for \cref{eq:unicycleModel} with $\alpha\in(0,1)$ and $\sigma_1,\ldots,\sigma_N\colon\mathbb{R}^5\to(0,\infty)$ as in \Cref{def:EGDCLF}. The subsequent \Cref{cond:1,cond:2} provide two examples of possible choices of $V$, $\alpha$, and $\sigma_1,\ldots,\sigma_N$.%
\begin{condition}\label{cond:1}
Choose $N\in\mathbb{N}$, $V\colon\mathbb{R}^5\to\mathbb{R}$, $\alpha\in(0,1)$, and $\sigma_1,\ldots,\sigma_N\colon\mathbb{R}^5\to(0,\infty)$ such that the following three conditions \ref{cond:1:1}--\ref{cond:1:3} are satisfied.\vspace{-0.3cm}

\begin{enumerate}[label=(C\arabic*),leftmargin=0.8cm]
	\item\label{cond:1:1} $N\geq8$.
    \item\label{cond:1:2} The function $V$ is given by%
    \begin{equation}\label{cond:1:eq:1}
    V(x) \ = \ \big(\mathrm{x}^4 + |\mathrm{y}|^{\phi(|\mathrm{y}|)} + \theta^4 + v^4 + \omega^4\big)^{1/4},
    \end{equation}%
    where $\phi\colon[0,\infty)\to[2,4]$ is non-decreasing, and satisfies $\phi(s)=2$ for $s\leq1$ and $\phi(s)=4$ for $s\geq2$.
    \item\label{cond:1:3} The functions $\sigma_1,\ldots,\sigma_N$ are constant and satisfy%
    \begin{subequations}
    \begin{equation}
    \sigma_i \ \leq \ \tfrac{1-\alpha}{(N-1)\,c}
    \end{equation}
    for every $i\in\{1,\ldots,N-1\}$ and
    \begin{equation}
    \sigma_N \ = \ 1 - (\sigma_1 + \cdots + \sigma_{N-1}),
    \end{equation}
    \end{subequations}
    where
    \begin{equation}
    \begin{split}
    c \ := \ \max\{& 32(3+\tfrac{16\cdot 7^4}{h^4\max(7,N-8)^4}),\; \\
    &(24h^4+\tfrac{128\cdot 7^4}{\max(7,N-8)^4})\}^\frac14. 
    \end{split}
    \end{equation}
\end{enumerate}%
\end{condition}%
The function $V$ in \cref{cond:1:eq:1} is inspired by the running cost function in \cite{Worthmann2016} for MPC for a kinematic unicycle. For states $x\in\mathbb{R}^5$ with component $|\mathrm{y}|\leq1$, the function $V$ in \cref{cond:1:eq:1} coincides with the running cost function from \cite{Worthmann2016}. The function from \cite{Worthmann2016} is well-suited for semi-global stabilization but not for global stabilization. To circumvent this problem, we introduce the state-dependent exponent function $\phi$ in \ref{cond:1:2} of \Cref{cond:1}. For states $x\in\mathbb{R}^5$ with component $|\mathrm{y}|\geq2$, the function $V$ in \cref{cond:1:eq:1} coincides with the $4$-norm on $\mathbb{R}^5$. The function $\phi$ provides a transition between the exponents $2$ and $4$ on the compact interval $[1,2]$. A transition from $2$ to $4$ on an interval different than $[1,2]$ would also be possible.%
\begin{proposition}\label{prop:cond:1}
Choose $N\in\mathbb{N}$, $V\colon\mathbb{R}^5\to\mathbb{R}$, $\alpha\in(0,1)$, and $\sigma_1,\ldots,\sigma_N\colon\mathbb{R}^5\to(0,\infty)$ according to \Cref{cond:1}. Then $V$ is an EGDCLF of order $N$ for \cref{eq:unicycleModel} with $\alpha$ and $\sigma_1,\ldots,\sigma_N$ as in \Cref{def:EGDCLF}. Moreover, there exist constants $\xi_1,\xi_2,\zeta>0$ such that%
\begin{equation}\label{prop:cond:1:eq}%
\begin{split}
\chi_1(r) & \ := \ \xi_1\,r, \qquad \chi_2(r) \ := \ \xi_2\,\max\{r^{1/2},r\}, \\
\varphi(r) & \ := \ \zeta\,r
\end{split}%
\end{equation}%
define $\mathcal{K}_\infty$-functions $\chi_1,\chi_2,\varphi$ for which \cref{eq:EGDCLF:1} and \cref{eq:EGDCLF:4} are satisfied for every $x\in\mathbb{R}^5$.
\end{proposition}%
A proof of \Cref{prop:cond:1} can be found in \hyperlink{sec:C}{Appendix~C}. As an alternative to \Cref{cond:1}, another example of a suitable choice of $V$, $\alpha$, and $\sigma_1,\ldots,\sigma_N$ is provided in the subsequent \Cref{cond:2}.%
\begin{condition}\label{cond:2}
Choose $N\in\mathbb{N}$, $V\colon\mathbb{R}^5\to\mathbb{R}$, $\alpha\in(0,1)$, and $\sigma_1,\ldots,\sigma_N\colon\mathbb{R}^5\to(0,\infty)$ such that the following three conditions \ref{cond:2:1}--\ref{cond:2:3} are satisfied.\vspace{-0.3cm}

\begin{enumerate}[label=(C\arabic*),leftmargin=0.8cm]
	\item\label{cond:2:1} $N\geq8$.
  \item\label{cond:2:2} The function $V$ is given by%
  \begin{equation}\label{cond:2:eq:1}
  V(x) \ = \ \|x\|.
  \end{equation}%
  \item\label{cond:2:3} The functions $\sigma_1,\ldots,\sigma_N$ are given by%
  \begin{subequations}\label{cond:2:eq:2}%
  \begin{equation}\label{cond:2:eq:2:a}
  \!\!\!\!\!\! \sigma_i(x) \ = \  \tfrac{1-\alpha}{(N-1)\,c}\cdot\left\{ \begin{tabular}{cl} $\!\!\|x\|^{\frac{1}{2}}\!\!$ & if $\ 0<\|x\|<1$ \\ 1 & otherwise \end{tabular} \right.
  \end{equation}%
  for every $i\in\{1,\ldots,N-1\}$ and by%
  \begin{equation}\label{cond:2:eq:2:b}
  \sigma_N(x) \ = \ 1 - \big(\sigma_1(x) + \cdots + \sigma_{N-1}(x)\big),
  \end{equation}%
  \end{subequations}
  where
  \begin{equation}
  c \ := \ 2\,(1+h^2)^{\frac{1}{2}}\,\big(\tfrac{9}{4}+\tfrac{4}{h^2\max\{1,(N-8)/7\}^2}\big)^{\frac{1}{2}}.
  \end{equation}
\end{enumerate}%
\end{condition}%
The function $V$ in \cref{cond:2:eq:1} is just the norm of the system state of the dynamic unicycle model \cref{eq:unicycleModel}. Such a very simple choice of $V$ is possible due to the state dependence of the weights $\sigma_1,\ldots,\sigma_N$. The first $N-1$ weights in \cref{cond:2:eq:2:a} become smaller with decreasing distance to the origin. This novel feature gives the unicycle more freedom to move from a ``critical'' initial state like \cref{eq:noGES} in \Cref{prop:noGES} towards the origin within the prediction horizon.%
\begin{proposition}\label{prop:cond:2}
Choose $N\in\mathbb{N}$, $V\colon\mathbb{R}^5\to\mathbb{R}$, $\alpha\in(0,1)$, and $\sigma_1,\ldots,\sigma_N\colon\mathbb{R}^5\to(0,\infty)$ according to \Cref{cond:2}. Then $V$ is an EGDCLF of order $N$ for \cref{eq:unicycleModel} with $\alpha$ and $\sigma_1,\ldots,\sigma_N$ as in \Cref{def:EGDCLF}. Moreover, there exist constants $\xi_1,\xi_2,\zeta>0$ such that%
\begin{equation}\label{prop:cond:2:eq}%
\begin{split}
\chi_1(r) & \ := \ \xi_1\,r, \qquad \chi_2(r) \ := \ \xi_2\,r, \\
\varphi(r) & \ := \ \zeta\,\max\{r^{1/2},r\}
\end{split}%
\end{equation}%
define $\mathcal{K}_\infty$-functions $\chi_1,\chi_2,\varphi$ for which \cref{eq:EGDCLF:1} and \cref{eq:EGDCLF:4} are satisfied for every $x\in\mathbb{R}^5$.
\end{proposition}%
A proof of \Cref{prop:cond:2} can be found in \hyperlink{sec:C}{Appendix~C}.%
\begin{remark}
The choice of parameters according to \Cref{cond:1} or \Cref{cond:2} depends only on the step size $h>0$ but not on the physical constants $m$, $J$, $k$, and $\kappa$. This independence of the components from physical constants may be an advantageous feature for future extensions to flexible-step MPC for systems with unknown system parameters as, e.g., in \cite{Pietschner2026} for linear systems.
\end{remark}%
\Cref{cond:1,cond:2} offer two possible ways for the choice of parameters among many others. For these two choices of parameters, we can now directly conclude that flexible-step MPC renders the unicycle globally $\gamma$-exponentially stable in the sense of \Cref{def:GES}.%
\begin{corollary}\label{corollary}
Let $f_0\colon\mathbb{R}^5\times\mathbb{R}^2\to\mathbb{R}$, $F\colon\mathbb{R}^5\to\mathbb{R}$. Choose $N\in\mathbb{N}$, $V\colon\mathbb{R}^5\to\mathbb{R}$, $\alpha\in(0,1)$, $\sigma_1,\ldots,\sigma_N\colon\mathbb{R}^5\to(0,\infty)$ according to \Cref{cond:1} or \Cref{cond:2}. Then, the optimal control problem \cref{eq:generalOCP} with $f\colon\mathbb{R}^5\times\mathbb{R}^2\to\mathbb{R}^5$ defined by \cref{eq:unicycleMap} is feasible for every initial state $x_0\in\mathbb{R}^5$. Moreover, the unicycle model \cref{eq:unicycleModel} controlled by \Cref{algo:fsMPC} is globally $\gamma$-exponentially stable, where $\gamma\in\mathcal{K}_\infty$ is defined by \cref{eq:max}.
\end{corollary}%
\begin{pf}
The statement is a direct consequence of \Cref{thm:expStab} and \Cref{prop:cond:1,prop:cond:2}.\hfill$\Box$
\end{pf}%
Recall from \Cref{sec:3} that the $\mathcal{K}_\infty$-function $\gamma$ defined in \cref{eq:max} describes the ``best'' notion of global $\gamma$-exponential stability that one can expect for the unicycle model \cref{eq:unicycleModel}. \Cref{corollary} states that this optimal stability notion can be established by flexible-step MPC. In the next section, we show through numerical simulations that the proposed method is not limited to the discretized unicycle model \cref{eq:unicycleModel}, but can also be successfully applied to the force- and torque-controlled continuous-time unicycle \cref{eq:unicycleModelContinuousTime} even when the step size $h$ is large.

%----------------------------------------------------------------------
%
%                               Section 6
%
%----------------------------------------------------------------------

\section{Simulation results}\label{sec:6}
To illustrate the closed-loop behavior of Algorithm~1 under the parameter choices of Conditions~\ref{cond:1} and~\ref{cond:2}, we present two simulations, one for each condition and with a different initial state. The plant is assumed to evolve according to the continuous-time dynamics~\cref{eq:unicycleModelContinuousTime}, whereas the prediction model is taken as in~\cref{eq:unicycleModel} with \(h=1\).

In both experiments, we consider soft state constraints induced by two prohibited ellipsoidal regions in the \((\mathrm x,\mathrm y)\)-plane. These regions are incorporated into the stage cost via penalty terms. More precisely, we choose
\[
f_0(x,u)
:=
\left\|[
\begin{smallmatrix}
\mathrm{x}\\
\mathrm{y}
\end{smallmatrix}]
\right\|^2
+
\|u\|^2
+
\sum_{i=1}^2 \rho \bigl(\ell_i^2+\ell_i\bigr),
\]
where
$
\ell_i
:=
\max\!\left\{
-(p-p_i)^\top Q_i (p-p_i)+1,\;0
\right\},
\ i\in\{1,2\}.
$
For both experiments, the parameters $ m=10,\ J=20,\ k=5,\ \kappa=15,\ \alpha = 0.3,\ \rho = 10^5,\ N = 12,\ p_1=[\begin{smallmatrix}
    6\\7
\end{smallmatrix}], \ p_2=[\begin{smallmatrix}
    4\\0
\end{smallmatrix}],\ Q_1=[\begin{smallmatrix}
    0.1 & 0\\0&0.4
\end{smallmatrix}],\ Q_2=\tfrac{1}{7}[\begin{smallmatrix}
    3.5 & 0\\0&0.6
\end{smallmatrix}]$ are chosen to be the same.

The first simulation is taken according to Condition~\ref{cond:1} with $\sigma_i=10^{-3}, \ i\in [1:11]$ and the EGDCLF $V_1$ is taken such that $\phi(s)=12s^5-90s^4+260s^3-360s^2+240s-60,\ s\in[1,2]$. The initial state is $x_0^{(1)} = [15,\ 15,\ -\tfrac{\pi}{4},\ 0,\ 0]^\top$. The second simulation is taken according to Condition~\ref{cond:2}, where we denote the EGDCLF $V_2=\|x\|$. The initial state is $x_0^{(2)}=[10,\ 0, \ \tfrac{\pi}{2},\ -3,0]^\top$.
Different initial conditions are chosen, as both parametrizations behave similarly for large state norms, while their differences are primarily visible near the origin.
\Cref{fig:sim} illustrates the resulting closed-loop trajectories together with the prohibited regions. In both cases, we observe that the proposed flexible-step MPC scheme steers the system toward the origin while successfully avoiding the penalized regions. The use of the discretized prediction model seems not to have a significant effect on the closed-loop behavior. In general we have observed that parametrizations according to Condition~\ref{cond:1} lead to smoother behavior around the origin, whereas Condition~\ref{cond:2} typically results in more pronounced fluctuations.
\begin{figure}[h]
\centering\includegraphics[scale=0.85]{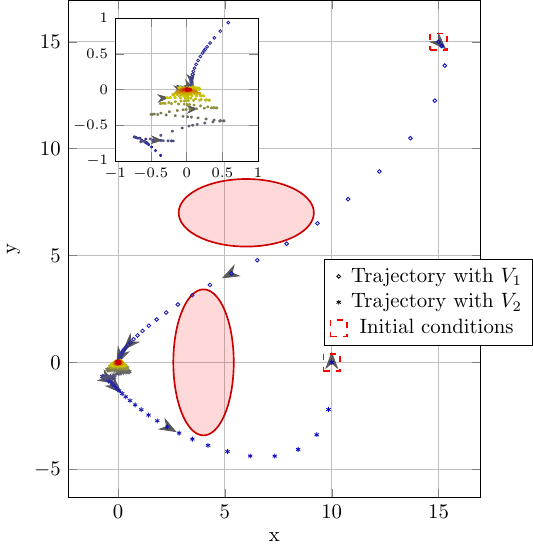}
\caption{Closed-loop trajectories of flexible-step MPC with the \emph{continuous-time} dynamic unicycle \cref{eq:unicycleModelContinuousTime} in
the presence of soft state constraints. The direction of each arrow indicates the orientation of the unicycle. In the simulation with EGDCLF $V_2$, the unicycle is moving ``backwards'' towards the origin.}
\label{fig:sim}
\end{figure}

\section{Conclusions and outlook}
We have investigated global exponential stabilization of the force- and torque-actuated unicycle model with flexible-step MPC. This nonholonomic system cannot be exponentially stabilized in the classical sense, as a high overshoot turns to be necessary for initial conditions in arbitrary small neighborhoods. Our algorithm guarantees the smallest possible overshoot. 
For future research, it might be promising to investigate a systematic approach to parametrize flexible-step MPC for other nonlinear systems. Finally, the incorporation of input and state constraints into the present framework appears to be a natural next step.
\bibliographystyle{IEEEtran}
\bibliography{bibFile}

%----------------------------------------------------------------------
%
%                               Appendix
%
%----------------------------------------------------------------------

\appendix
\section{Proof of \texorpdfstring{\Cref{prop:noGES}}{Proposition~\ref{prop:noGES}}}\label{sec:A}
Let $r$, $q$, $\lambda$, $\mu$ be positive real numbers with $1/2<q\leq1$. For every $z\in\mathbb{R}$, let $\lceil{z}\rceil$ denote the smallest integer $\geq{z}$. Because $1/2<q\leq1$, we can find some sufficiently small $\varepsilon\in(0,r)$ such that
\begin{equation}\label{eq:proofNoGES:1}
\underbrace{\left\lceil - \tfrac{1}{\mu}\,\log\big(\tfrac{1}{2\,\lambda}\,\varepsilon^{1-q}\big) \right\rceil}_{=:\tau}h\,\lambda^2\,\varepsilon^{2\,q-1} \ < \ \tfrac{1}{2}.
\end{equation}
For the sake of reaching a contradiction, suppose that there exists $u\colon\mathbb{N}\to\mathbb{R}^2$ such that the solution $x\colon\mathbb{N}\to\mathbb{R}^5$ of \cref{eq:unicycleModel} with input $u$ and initial condition $x(0)=[0,\varepsilon,0,0,0]^\top$ satisfies inequality \cref{eq:MichelBook3} for every $t\in\mathbb{N}$. Then, in particular,
\begin{equation}\label{eq:proofNoGES:2}
\|x(t)\| \ \leq \ \lambda\,\varepsilon^q
\end{equation}
for every $t\in\mathbb{N}$ and also
\begin{equation}\label{eq:proofNoGES:3}
\|x(\tau)\| \ \leq \ \|x(0)\|^q\,\lambda\,\mathrm{e}^{-\mu\,\tau} \ \leq \ \tfrac{\varepsilon}{2},
\end{equation}
where we have used \cref{eq:proofNoGES:1}. We conclude from \cref{eq:proofNoGES:2}, \cref{eq:proofNoGES:1} that
\begin{equation}\label{eq:proofNoGES:4}
\Big|\sum_{s=0}^{\tau-1}h\,v(s)\,\sin(\theta(s))\Big| \ \leq \ \tau\,h\,\lambda^2\,\varepsilon^{2q} \ < \ \tfrac{\varepsilon}{2}.
\end{equation}
Since $x$ is a solution of \cref{eq:unicycleModel}, it follows that
\begin{align}
\|x(t)\| & \ \geq \ |\mathrm{y}(t)| \ = \ \Big|\mathrm{y}(0) + \sum_{s=0}^{t-1}h\,v(s)\,\sin(\theta(s))\Big| \nonumber \\
& \ \geq \ \Big|\varepsilon - \Big|\sum_{s=0}^{t-1}h\,v(s)\,\sin(\theta(s))\Big|\Big| \ > \ \tfrac{\varepsilon}{2}, \label{eq:proofNoGES:5}
\end{align}
which contradicts \cref{eq:proofNoGES:3}.

\section{Proof of \texorpdfstring{\Cref{thm:expStab}}{Theorem~\ref{thm:expStab}}}\label{sec:B}
Assume that $V$ is an EGDCLF of order $N$ for \cref{eq:generalControlSystem} with $\alpha\in(0,1)$, $\sigma_1,\ldots,\sigma_N\colon\mathbb{X}\to(0,\infty)$, $\chi_1,\chi_2,\varphi\in\mathcal{K}_\infty$, and $\tilde{\lambda},\tilde{\mu}>0$ as in \Cref{def:EGDCLF}. Because of property \ref{def:EGDCLF:2} in \Cref{def:EGDCLF}, the optimal control problem \cref{eq:generalOCP} is feasible for every initial state $x_0\in\mathbb{X}$. For a proof of \Cref{thm:expStab}, it remains to be shown that system \cref{eq:generalControlSystem} controlled by \Cref{algo:fsMPC} is globally $\gamma$-exponentially stable, where $\gamma:=\chi_1^{-1}\circ\varphi\circ\chi_2\in\mathcal{K}_\infty$. To this end, set $\lambda:=\tilde{\lambda}\,\mathrm{e}^{\tilde{\mu}\alpha}$, and $\mu:=\alpha\,\tilde{\mu}/N$. Fix an arbitrary state--input trajectory $(x,u)\colon\mathbb{N}\to\mathbb{X}\times\mathbb{U}$ of system \cref{eq:generalControlSystem} controlled by \Cref{algo:fsMPC}. The proof is complete if we can show that inequality \cref{eq:MichelBook2} holds at every time $t\in\mathbb{N}$.

Let $0=\tau_0<\tau_1<\cdots$ denote all time instants at which step \ref{algo:fsMPC:1} in \Cref{algo:fsMPC} is executed. Because of step \ref{algo:fsMPC:2} in \Cref{algo:fsMPC} and property \ref{def:EGDCLF:3} in \Cref{def:EGDCLF}, we have
\begin{equation}\label{eq:proofExpStab:1}
(1-\alpha)\,V(x(\tau_k)) \ \leq \ \varphi(V(x(\tau_k)))\cdot\min_{i\in[1:N]}\sigma_i(x(\tau_k))
\end{equation}
for every $k\in\mathbb{N}$. Because of step \ref{algo:fsMPC:3} in \Cref{algo:fsMPC}, we have
\begin{equation}\label{eq:proofExpStab:2}
V(x(\tau_{k+1})) \ \leq \ (1-\alpha)\,V(x(\tau_k))
\end{equation}
for every $k\in\mathbb{N}$. Because of step \ref{algo:fsMPC:4} in \Cref{algo:fsMPC} and the average descent condition \ref{def:EGDCLF:2} in \Cref{def:EGDCLF}, we have
\begin{equation}\label{eq:proofExpStab:3}
V(x(t))\cdot\min_{i\in[1:N]}\sigma_i(x(\tau_k)) \ \leq \ (1-\alpha)\,V(x(\tau_k))
\end{equation}
for every $k\in\mathbb{N}$ and every $t\in(\tau_k\!:\!\tau_{k+1}]$.

Fix an arbitrary $k\in\mathbb{N}$ and an arbitrary $t\in(\tau_k\!:\!\tau_{k+1}]$. Because $\sigma_1,\ldots,\sigma_N$ are assumed to attain only positive values, it follows from \cref{eq:proofExpStab:1} and \cref{eq:proofExpStab:3} that
\begin{equation}\label{eq:proofExpStab:4}
V(x(t)) \ \leq \ \varphi(V(x(\tau_k))).
\end{equation}
Using $t\leq{(k+1)N}$ and the known inequality $1-\alpha\leq\mathrm{e}^{-\alpha}$, we conclude from \cref{eq:proofExpStab:2} and \cref{eq:proofExpStab:4} that
\begin{equation}\label{eq:proofExpStab:5}
V(x(t)) \ \leq \ \varphi\big(V(x_0)\,\mathrm{e}^{\alpha-\alpha\,t/N}\big).
\end{equation}
Because of property \ref{def:EGDCLF:1} in \Cref{def:EGDCLF}, it follows that
\begin{equation}\label{eq:proofExpStab:6}
\|x(t)\| \ \leq \ \big(\chi_1^{-1}\circ\varphi\big)\big(\chi_2(\|x_0\|)\,\mathrm{e}^{\alpha-\alpha\,t/N}\big)
\end{equation}
Now the asserted inequality \cref{eq:MichelBook2} follows from property \ref{def:EGDCLF:4} in \Cref{def:EGDCLF} and the suitable definitions of $\gamma$, $\lambda$, and $\mu$.

\section{Proofs of \texorpdfstring{\Cref{prop:cond:1,prop:cond:2}}{Propositions~\ref{prop:cond:1} and~\ref{prop:cond:2}}}\label{sec:C}
\subsection{Steering the unicycle to the origin}\label{sec:C.1}
Fix an arbitrary initial state%
\begin{equation}\label{Ala:01}
x_0 \ = \ \big[ \ \mathrm{x}_0, \ \mathrm{y}_0, \ \theta_0, \ v_0, \ \omega_0 \ \big]^\top \in \mathbb{R}^5.
\end{equation}%
We will construct a suitable state-input trajectory $(x,u)$, which describes the motion of the unicycle \cref{eq:unicycleModel} from $x_0$ to the origin within a prescribed prediction horizon $N\geq8$. As abbreviations, set%
\begin{equation}
\begin{split}
\mathrm{x}_1 & \ := \ \mathrm{x}_0 + h\,v_0\,\cos(\theta_0), \qquad \mathrm{\theta}_1 \ := \ \theta_0 + h\,\omega_0, \allowdisplaybreaks \label{Ala:02} \\
\mathrm{y}_1 & \ := \ \mathrm{y}_0 + h\,v_0\,\sin(\theta_0).
\end{split}
\end{equation}%
Choose any $\bar{\theta}\in[-\frac{\pi}{4},\frac{\pi}{4}]$ and $d\geq0$ such that%
\begin{equation}\label{Ala:04}
\bar{x} \ := \ d\,\cos(-\bar{\theta}), \qquad \bar{\mathrm{y}} \ := \ \tfrac{\mathrm{y}_1}{2} \ = \ \mathrm{y}_1 + d\,\sin(-\bar{\theta}).
\end{equation}%
Then $\bar{\mathrm{x}}^2+\bar{\mathrm{y}}^2=d^2$. Let $T$ be the largest integer less than or equal to $(N-1)/7$. Set $t_i:=1+(i-1)T$ for every $i\in[1\!:\!8]$. Define $x=[ \mathrm{x}, \mathrm{y}, \theta, v, \omega ]^\top\colon[0\!:\!N]\to\mathbb{R}^5$ by%
\begin{align}
x(t) & \ := \ x_0 & \text{for} \ t & = 0, \nonumber\allowdisplaybreaks \\
x(t) & \ := \ \left[ \begin{smallmatrix} \mathrm{x}_1 \\ \mathrm{y}_1 \\ \theta_1 + \frac{t-t_1}{T}(-\theta_1) \\ 0 \\ \frac{-\theta_1}{hT} \end{smallmatrix} \right] & \text{for} \ t & \in [\vphantom{]}t_1\!:\!t_2\vphantom{(}), \nonumber\allowdisplaybreaks \\
x(t) & \ := \ \left[ \begin{smallmatrix} \mathrm{x}_1 + \frac{t-t_2}{T}\,(-\mathrm{x}_1)\,\cos(0) \\ \mathrm{y}_1 + \frac{t-t_2}{T}\,(-\mathrm{x}_1)\,\sin(0) \\ 0 \\ \frac{-\mathrm{x}_1}{hT} \\ 0 \end{smallmatrix} \right] & \text{for} \ t & \in [\vphantom{]}t_2\!:\!t_3\vphantom{(}), \nonumber\allowdisplaybreaks \\
x(t) & \ := \ \left[ \begin{smallmatrix} 0 \\ \mathrm{y}_1 \\ 0 + \frac{t-t_3}{T}(-\bar{\theta}) \\ 0 \\ \frac{-\bar{\theta}}{hT} \end{smallmatrix} \right] & \text{for} \ t & \in [\vphantom{]}t_3\!:\!t_4\vphantom{(}), \nonumber\allowdisplaybreaks \\
x(t) & \ := \ \left[ \begin{smallmatrix} 0 + \frac{t-t_4}{T}\,d\,\cos(-\bar{\theta}) \\ \mathrm{y}_1 + \frac{t-t_4}{T}\,d\,\sin(-\bar{\theta}) \\ - \bar{\theta} \\ \frac{d}{hT} \\ 0 \end{smallmatrix} \right] & \text{for} \ t & \in [\vphantom{]}t_4\!:\!t_5\vphantom{(}), \nonumber\allowdisplaybreaks \\
x(t) & \ := \ \left[ \begin{smallmatrix} \bar{\mathrm{x}} \\ \bar{\mathrm{y}} \\ -\bar{\theta} + \frac{t-t_5}{T}\,2\,\bar{\theta} \\ 0 \\ \frac{2\,\bar{\theta}}{hT} \end{smallmatrix} \right] & \text{for} \ t & \in [\vphantom{]}t_5\!:\!t_6\vphantom{(}), \nonumber\allowdisplaybreaks \\
x(t) & \ := \ \left[ \begin{smallmatrix} \bar{\mathrm{x}} + \frac{t-t_6}{T}\,(-d)\,\cos(\bar{\theta}) \\ \bar{\mathrm{y}} + \frac{t-t_6}{T}\,(-d)\,\sin(\bar{\theta}) \\ \bar{\theta} \\ \frac{-d}{hT} \\ 0 \end{smallmatrix} \right] & \text{for} \ t & \in [\vphantom{]}t_6\!:\!t_7\vphantom{(}), \nonumber\allowdisplaybreaks \\
x(t) & \ := \ \left[ \begin{smallmatrix} 0 \\ 0 \\ \bar{\theta} + \frac{t-t_7}{T}\,(-\bar{\theta}) \\ 0 \\ \frac{-\bar{\theta}}{hT} \end{smallmatrix} \right] & \text{for} \ t & \in [\vphantom{]}t_7\!:\!t_8\vphantom{(}), \nonumber\allowdisplaybreaks \\
x(t) & \ := \ 0 & \text{for} \ t & \in [t_8\!:\!N]. \label{Ala:05}
\end{align}%
Define $u=[ F, \mathcal{T} ]^\top\colon[\vphantom{]}0\!:\!N\vphantom{(})\to\mathbb{R}^2$ by%
\begin{equation}\label{Ala:06}
\begin{bmatrix} F(t) \\ \mathcal{T}(t) \end{bmatrix} \ := \ \begin{bmatrix} k\,v(t) + \frac{m}{h}\,(v(t+1)-v(t)) \\ \kappa\,\omega(t) + \frac{J}{h}\,(\omega(t+1)-\omega(t)) \end{bmatrix}.
\end{equation}%
It is easy to see that $(u,x)$ is an input--state trajectory of the unicycle \cref{eq:unicycleModel}.%

\subsection{Proof of \texorpdfstring{\Cref{prop:cond:1}}{Proposition~\ref{prop:cond:1}}}\label{sec:C.2}
Choose $N\in\mathbb{N}$, $V\colon\mathbb{R}^5\to\mathbb{R}$, $\alpha\in(0,1)$, as well as $\sigma_1,\ldots,\sigma_N\colon\mathbb{R}^5\to(0,\infty)$ according to \Cref{cond:1}. Set%
\begin{equation}\label{Christian:01}
\xi_1 \ := \ \tfrac{1}{20}, \qquad \xi_2 \ := \ 2, \qquad \zeta \ := \ \tfrac{1}{\min\{\sigma_1,\ldots,\sigma_N\}},
\end{equation}%
and define $\chi_1,\chi_2,\varphi\in\mathcal{K}_\infty$ by \cref{prop:cond:1:eq}. Then, it is straightforward to check that \cref{eq:EGDCLF:1} and \cref{eq:EGDCLF:4} are satisfied for every $x\in\mathbb{R}^5$. Moreover, if we set $\tilde{\lambda}:=\tilde{\mu}:=1$, then \cref{eq:EGDCLF:5} holds for every $r\geq0$ and every $\tau\geq-\alpha$. The proof of \Cref{prop:cond:1} is complete if we can show that property \ref{def:EGDCLF:2} in \Cref{def:EGDCLF} holds.

Define%
\begin{equation}\label{Bahman:05}
\begin{split}
\bar{\theta} & \ := \ \mathrm{sgn}(\mathrm{y}_1)\left\{ \begin{tabular}{cl} $\arctan(|\mathrm{y}_1|^{1/2})$ & if $|\mathrm{y}_1|\leq1$ \\ $\frac{\pi}{4}$ & if $|\mathrm{y}_1|>1$ \end{tabular} \right., \allowdisplaybreaks \\
d & \ := \ \left\{ \begin{tabular}{cl} $\frac{1}{2}\,|\mathrm{y}_1|^{1/2}\,(1+|\mathrm{y}_1|)^{1/2}$ & if $|\mathrm{y}_1|\leq1$ \\ $\frac{1}{2^{1/2}}\,|\mathrm{y}_1|$ & if $|\mathrm{y}_1|>1$ \end{tabular} \right..
\end{split}
\end{equation}%
Let $x\colon[0\!:\!N]\to\mathbb{R}^5$ be defined as in \Cref{sec:C.1}. The proof of \Cref{prop:cond:1} is complete if we can show that the average descent condition \cref{eq:EGDCLF:2} is satisfied for $x$. From the definitions in \cref{cond:1:eq:1} and \cref{Ala:05}, we get%
\begin{align}
V(x(t))^4 & \leq \mathrm{x}_1^4 + |\mathrm{y}_1|^{\phi(|\mathrm{y}_1|)} + \theta_1^4 + \tfrac{\theta_1^4}{(hT)^4} & \text{for} \ t & \in [\vphantom{]}t_1\!:\!t_2\vphantom{(}), \nonumber \allowdisplaybreaks \\
V(x(t))^4 & \leq \mathrm{x}_1^4 + |\mathrm{y}_1|^{\phi(|\mathrm{y}_1|)} + \tfrac{\mathrm{x}_1^4}{(hT)^4} & \text{for} \ t & \in [\vphantom{]}t_2\!:\!t_3\vphantom{(}), \nonumber \allowdisplaybreaks \\
V(x(t))^4 & \leq |\mathrm{y}_1|^{\phi(|\mathrm{y}_1|)} + \bar{\theta}^4 + \tfrac{\bar{\theta}^4}{(hT)^4} & \text{for} \ t & \in [\vphantom{]}t_3\!:\!t_4\vphantom{(}), \nonumber \allowdisplaybreaks \\
V(x(t))^4 & \leq \bar{\mathrm{x}}^4 + |\mathrm{y}_1|^{\phi(|\mathrm{y}_1|)} + \bar{\theta}^4 + \tfrac{d^4}{(hT)^4} & \text{for} \ t & \in [\vphantom{]}t_4\!:\!t_5\vphantom{(}), \nonumber \allowdisplaybreaks \\
V(x(t))^4 & \leq \bar{\mathrm{x}}^4 + |\bar{\mathrm{y}}|^{\phi(|\bar{\mathrm{y}}|)} + \bar{\theta}^4 + \tfrac{(2\bar{\theta})^4}{(hT)^4} & \text{for} \ t & \in [\vphantom{]}t_5\!:\!t_6\vphantom{(}), \nonumber \allowdisplaybreaks \\
V(x(t))^4 & \leq \bar{\mathrm{x}}^4 + |\bar{\mathrm{y}}|^{\phi(|\bar{\mathrm{y}}|)} + \bar{\theta}^4 + \tfrac{d^4}{(hT)^4} & \text{for} \ t & \in [\vphantom{]}t_6\!:\!t_7\vphantom{(}), \nonumber \allowdisplaybreaks \\
V(x(t))^4 & \leq \bar{\theta}^4 + \tfrac{\bar{\theta}^4}{(hT)^4} \qquad\qquad \text{for} \ t \in [\vphantom{]}t_7\!:\!t_8\vphantom{(}).\!\!\!\!\!\!\!\!\!\!\!\!\!\!\!\!\!\!\!\!\!\!\!\!\!\!\!\! & & \label{Christian:02}
\end{align}%
In the next step, we derive estimates for the terms on the right-hand side of \cref{Christian:02}. It follows from equation \cref{Ala:02} and the assumptions on the function $\phi$ in \Cref{cond:1} that%
\begin{equation}\label{Christian:03}
\begin{split}
& \mathrm{x}_1^4 \ \leq \ 8\,\mathrm{x}_0^4 + 8\,h^4\,v_0^4, \qquad \theta_1^4 \ \leq \ 8\,\theta_0^4 + 8\,h^4\,\omega_0^4, \\
& |\mathrm{y}_1|^{\phi(|\mathrm{y}_1|)} \ \leq \ 32\,|\mathrm{y}_0|^{\phi(|\mathrm{y}_0|)} + 8\,h^4\,v_0^4 + \theta_0^4.
\end{split}
\end{equation}%
Using the definitions in \cref{Ala:04}, \cref{Bahman:05}, and the assumptions on $\phi$, we also obtain%
\begin{equation}\label{Christian:05}
\max\{\bar{x}^4, \bar{\theta}^4, |\bar{\mathrm{y}}|^{\phi(|\bar{\mathrm{y}}|)}\} \ \leq \ |\mathrm{y}_1|^{\phi(|\mathrm{y}_1|)}. 
\end{equation}%
By applying the estimates in \cref{Christian:03}, \cref{Christian:05} to the terms on the right-hand side of \cref{Christian:02} and by using the definition of $V$ in \cref{cond:1:eq:1}, it follows that
\begin{equation}
V(x(t))^4 \ \leq \ c^4\,V(x_0)^4
\end{equation}%
for every $t\in[\vphantom{]}t_1\!:\!t_8\vphantom{(})$, where the constant $c$ is defined according to \Cref{cond:1}. Because%
\begin{align}
\sum_{t=1}^N\sigma_t(x_0)\,V(x(t)) & \ \leq \ \sum_{t=1}^{t_8-1}\sigma_t(x_0)\,c\,V(x_0) \nonumber \\
& \ \leq \ (1-\alpha)\,V(x_0),
\end{align}%
the average descent condition \cref{eq:EGDCLF:2} is satisfied.

\subsection{Proof of \texorpdfstring{\Cref{prop:cond:2}}{Proposition~\ref{prop:cond:2}}}\label{sec:C.3}
Choose $N\in\mathbb{N}$, $V\colon\mathbb{R}^5\to\mathbb{R}$, $\alpha\in(0,1)$, as well as $\sigma_1,\ldots,\sigma_N\colon\mathbb{R}^5\to(0,\infty)$ according to \Cref{cond:2}. Set%
\begin{equation}\label{Bahman:01}
\xi_1 \ := \ 1, \qquad \xi_2 \ := \ 1, \qquad \zeta \ := \ (N-1)\,c,
\end{equation}%
and define $\chi_1,\chi_2,\varphi\in\mathcal{K}_\infty$ by \cref{prop:cond:2:eq}. Then, it is straightforward to check that \cref{eq:EGDCLF:1} and \cref{eq:EGDCLF:4} are satisfied for every $x\in\mathbb{R}^5$. Moreover, if we set $\tilde{\lambda}:=e^\frac{\alpha}{2}$ and $\tilde{\mu}:=\tfrac{1}{2}$, then \cref{eq:EGDCLF:5} holds for every $r\geq0$ and every $\tau\geq-\alpha$. The proof of \Cref{prop:cond:2} is complete if we can show that property \ref{def:EGDCLF:2} in \Cref{def:EGDCLF} holds.

Let $x\colon[0\!:\!N]\to\mathbb{R}^5$ be defined as in \Cref{sec:C.1}. The proof of \Cref{prop:cond:2} is complete if we can show that the average descent condition \cref{eq:EGDCLF:2} is satisfied for $x$. From the definitions in \cref{cond:2:eq:1} and \cref{Ala:05}, we get%
\begin{align}
V(x(t))^2 & \leq \mathrm{x}_1^2 + \mathrm{y}_1^2 + \theta_1^2 + \tfrac{\theta_1^2}{(hT)^2} & \text{for} \ t & \in [\vphantom{]}t_1\!:\!t_2\vphantom{(}), \nonumber \allowdisplaybreaks \\
V(x(t))^2 & \leq \mathrm{x}_1^2 + \mathrm{y}_1^2 + \tfrac{\mathrm{x}_1^2}{(hT)^2} & \text{for} \ t & \in [\vphantom{]}t_2\!:\!t_3\vphantom{(}), \nonumber \allowdisplaybreaks \\
V(x(t))^2 & \leq \mathrm{y}_1^2 + \bar{\theta}^2 + \tfrac{\bar{\theta}^2}{(hT)^2} & \text{for} \ t & \in [\vphantom{]}t_3\!:\!t_4\vphantom{(}), \nonumber \allowdisplaybreaks \\
V(x(t))^2 & \leq \bar{\mathrm{x}}^2 + \mathrm{y}_1^2 + \bar{\theta}^2 + \tfrac{d^2}{(hT)^2} & \text{for} \ t & \in [\vphantom{]}t_4\!:\!t_5\vphantom{(}), \nonumber \allowdisplaybreaks \\
V(x(t))^2 & \leq \bar{\mathrm{x}}^2 + \bar{\mathrm{y}}^2 + \bar{\theta}^2 + \tfrac{(2\bar{\theta})^2}{(hT)^2} & \text{for} \ t & \in [\vphantom{]}t_5\!:\!t_6\vphantom{(}), \nonumber \allowdisplaybreaks \\
V(x(t))^2 & \leq \bar{\mathrm{x}}^2 + \bar{\mathrm{y}}^2 + \bar{\theta}^2 + \tfrac{d^2}{(hT)^2} & \text{for} \ t & \in [\vphantom{]}t_6\!:\!t_7\vphantom{(}), \nonumber \allowdisplaybreaks \\
V(x(t))^2 & \leq \bar{\theta}^2 + \tfrac{\bar{\theta}^2}{(hT)^2} \qquad\qquad \text{for} \ t \in [\vphantom{]}t_7\!:\!t_8\vphantom{(}).\!\!\!\!\!\!\!\!\!\!\!\!\!\!\!\!\!\!\!\!\!\!\!\!\!\!\!\! & & \label{Bahman:02}
\end{align}%
In the next step, we derive estimates for the terms on the right-hand side of \cref{Bahman:02}. It follows from \cref{Ala:02} that%
\begin{equation}\label{Bahman:03}
\begin{split}
\mathrm{x}_1^2 & \ \leq \ 2\,\mathrm{x}_0^2 + 2\,h^2\,v_0^2, \qquad \theta_1^2 \ \leq \ 2\,\theta_0^2 + 2\,h^2\,\omega_0^2, \\
\mathrm{y}_1^2 & \ \leq \ 2\,\mathrm{y}_0^2 + 2\,h^2\,v_0^2.
\end{split}
\end{equation}
Now we make a case analysis. First, we consider the case $\|x_0\|\leq1$. In this case, we define $\bar{\theta}$ and $d$ as in \cref{Bahman:05}. For this choice, we get estimates%
\begin{equation}\label{Bahman:07}
\begin{split}
\bar{\theta}^2 & \ \leq \ |\mathrm{y}_1|, \qquad d^2 \leq \tfrac{1}{2}\,\max\{|\mathrm{y}_1|,|\mathrm{y}_1|^2\}, \\
\bar{\mathrm{x}}^2 & \ \leq \ \tfrac{1}{4}\,\max\{|\mathrm{y}_1|,|\mathrm{y}_1|^2\}.
\end{split}
\end{equation}%
Second, we consider the case $\|x_0\|>1$. In this case, we define $\bar{\theta}$ and $d$ by%
\begin{equation}
\bar{\theta} \ := \ \mathrm{sgn}(\mathrm{y}_1)\,\tfrac{\pi}{4}, \qquad d \ := \ \tfrac{1}{2^{1/2}}\,|\mathrm{y}_1|, \label{Bahman:09}
\end{equation}%
and we obtain that%
\begin{equation}
\bar{\theta}^2 \ \leq \ \|x_0\|^2, \qquad d^2 \ = \ 2\,\bar{x}^2 \ \leq \ \tfrac{1}{2}\,|\mathrm{y}_1|^2. \label{Bahman:10}
\end{equation}%
By applying the estimates in \cref{Bahman:03}, \cref{Bahman:07}, \cref{Bahman:10} to the terms on the right-hand side of \cref{Bahman:02}, it follows that%
\begin{equation}
V(x(t))^2 \ \leq \ c^2\,\max\{\|x_0\|,\|x_0\|^2\}
\end{equation}%
for every $t\in[\vphantom{]}t_1\!:\!t_8\vphantom{(})$, where the constant $c$ is defined according to \Cref{cond:2}. Because%
\begin{align}
& \sum_{t=1}^N\sigma_t(x_0)\,V(x(t)) \nonumber\allowdisplaybreaks \\
& \ \leq \ \sum_{t=1}^{t_8-1}\sigma_t(x_0)\,c\,\max\big\{\|x_0\|^{\frac{1}{2}},\|x_0\|\big\} \nonumber\allowdisplaybreaks \\
& \ \leq \ (1-\alpha)\,\|x_0\| \ = \ (1-\alpha)\,V(x_0),
\end{align}%
the average descent condition \cref{eq:EGDCLF:2} is satisfied.
\end{document}